\documentclass[12pt]{article}
\usepackage{amsmath}
\usepackage{graphicx}
\usepackage{enumerate}
\usepackage{natbib}
\usepackage{url} % not crucial - just used below for the URL 

\RequirePackage[OT1]{fontenc}
\RequirePackage{amsthm,amsmath,amsfonts,amssymb}
\RequirePackage[colorlinks,citecolor=blue,urlcolor=blue]{hyperref} 

\usepackage[latin1]{inputenc} % ÃÂ¥Ãâ¬ÃÂ¶:s
\usepackage{verbatim}

\usepackage{paralist}
\usepackage{pstool}
\usepackage{booktabs}
\usepackage{tikz}
\usepackage{mathtools}
\usepackage{bbm}

\newtheorem{Assumption}{Assumption}
\newtheorem{Proposition}{Proposition}

\newtheorem{Theorem}{Theorem}
\newtheorem{Corollary}{Corollary}

\newtheorem{example}{Example}

\newcommand{\bbN}{\mathbb{N}}
\newcommand{\bbP}{\mathbb{P}}

\newcommand{\bbR}{\mathbb{R}}

\newcommand{\cA}{\mathcal{A}}
\newcommand{\cB}{\mathcal{B}}
\newcommand{\cC}{\mathcal{C}}
\newcommand{\cD}{\mathcal{D}}

\newcommand{\cS}{\mathcal{S}}

\newcommand{\cW}{\mathcal{W}}
\newcommand{\cX}{\mathcal{X}}

\newcommand{\from}{\colon}
\newcommand{\rd}{\mathrm{d}}

\DeclareMathOperator{\Cov}{Cov}

\DeclareMathOperator*{\argmin}{arg\,min}

\newcommand{\proper}{\mathsf}

\newcommand{\pE}{\proper{E}}

\newcommand{\pN}{\proper{N}}
\newcommand{\mv}[1]{{\boldsymbol{\mathrm{#1}}}}

\newcommand{\md}{\ensuremath{\,\mathrm{d}}}

\graphicspath{{../tex/figs/}}

\usepackage{authblk}
\allowdisplaybreaks

\begin{document}

\def\spacingset#1{\renewcommand{\baselinestretch}%
{#1}\small\normalsize} \spacingset{1}
\date{}
\renewcommand\Affilfont{\small}
%%%%%%%%%%%%%%%%%%%%%%%%%%%%%%%%%%%%%%%%%%%%%%%%%%%%%%%%%%%%%%%%%%%%%%%%%%%%%%

  \title{\bf Spatial confounding under infill asymptotics}
  \author[1]{David Bolin}
  \affil[1]{Statistics Program, Computer, Electrical and Mathematical Sciences and Engineering Division, King Abdullah University of Science and Technology (KAUST), Thuwal, Saudi Arabia}
    \author[2]{Jonas Wallin}
    \affil[2]{Department of Statistics, Lund University, Lund, Sweden}
  \maketitle

\begin{abstract}
The estimation of regression parameters in spatially referenced data plays a crucial role across various scientific domains. A common approach involves employing an additive regression model to capture the relationship between observations and covariates, accounting for spatial variability not explained by the covariates through a Gaussian random field.  While theoretical analyses of such models have predominantly focused on prediction and covariance parameter inference, recent attention has shifted towards understanding the theoretical properties of regression coefficient estimates, particularly in the context of spatial confounding. 
This article studies the effect of misspecified covariates, in particular when the misspecification changes the smoothness. We analyze the theoretical properties of the generalize least-square estimator under infill asymptotics, and show that the estimator can have counter-intuitive properties. 
In particular, the estimated regression coefficients can converge to zero as the number of observations increases, despite high correlations between observations and covariates. Perhaps even more surprising, the estimates can diverge to infinity under certain conditions.  Through an application to temperature and precipitation data, we show that both behaviors can be observed for real data. Finally, we propose a simple fix to the problem by adding a smoothing step in the regression.
\end{abstract}

\noindent%

\section{Introduction}
The estimation of regression parameters, sometimes also referred to as fixed effects, for spatially referenced data is of fundamental importance for practical applications in geostatistics, climate, disease mapping, brain imaging, and many other fields. A simple and commonly used model in these areas is the following additive regression model for observations $Y_1,\ldots, Y_n$ at spatial locations $s_1, \ldots, s_n$ in some domain $D$,
\begin{equation}\label{eq:main_model}
	Y_i = \sum_{k=1}^K X_k(s_i) \beta_k + \epsilon(s_i) = \mv{X}(s_i)\mv{\beta} + \epsilon(s_i),
\end{equation}
where $X_1,\ldots,X_K$ are covariates, $\beta_1, \ldots, \beta_K$ regression coefficients, and $\epsilon$ is a centered Gaussian random field that is supposed to capture the spatial variation in the data which is not captured by the covariates. 

The maximum likelihood estimate of the regression coefficients (assuming that the covariance function of $\epsilon$ is known) is given by the generalized least-squares estimator 
\begin{equation}\label{eq:gls}
\hat{\beta}_{GLS} = (\mv{X}^{\top}\mv{\Sigma}^{-1}\mv{X})^{-1}\mv{X}^{\top} \mv{\Sigma}^{-1}\mv{Y},
\end{equation}
where $\mv{Y} = (Y_i,\ldots, Y_n)^\top$, $\mv{X}$ is an $n \times K$ matrix with entries $X_{ik} = X_k(s_i)$ and $\mv{\Sigma}$ is an $n \times n$ matrix with elements $\mv{\Sigma}_{ij} = \Cov(\epsilon(s_i),\epsilon(s_j))$. This estimator is typically preferred over the ordinary least squares estimator for $\mv{\beta}$ (obtained by replacing $\mv{\Sigma} $ with the identity matrix). 
Models of this type have mostly been analyzed theoretically with a focus on prediction  or inference of covariance parameters \citep[e.g.,][]{stein99,Zhang2004,ibragimov2012gaussian,kirchner2022necessary}.
As noted by \cite{khan2022restricted}, there has until recently been few attempts to understand the theoretical properties of the regression coefficient estimates in spatial models like this. However, in the last few years several articles have investigated so-called spatial confounding problems  \cite[see][for an overview of the literature]{khan2023re}. \cite{khan2023re} note that there is no single definition of spatial confounding, but that one can classify spatial confounding in two broad categories of  model analysis confounding and data generation confounding depending on the assumptions that are made. In short,  data generation confounding assumes that the model \eqref{eq:main_model} is misspecified and what impact this has on \eqref{eq:gls}. Pioneering work in this area is \cite{paciorek2010importance} which studied the case when $\epsilon$ and $X$ are generated by a bivariate Gaussian random field.  On the other hand, model analysis confounding does not assume a specific misspecification but rather studies the relation between $\mv{\Sigma}$ and $\mv{X}$ and how this effects \eqref{eq:gls} and how its expected value deviates from the expected value of the OLS estimator. Early work in this area is \cite{reich2006effects} which studied the case when $\mv{\Sigma}$ is generated by a conditionally autoregressive model for disease-mapping models.

Despite the recent interest in spatial confounding, there are some crucial problems with the spatial regression model which seem to have been widely overlooked in the literature, and this is the main focus of this work.
Specifically, we are interested in the scenario when the smoothness of the observed process $Y$ is different from an influential covariate $X$. To model this, we suppose that the true data generating process is 
$$
Y(s) = S X(s) + \epsilon(s), 
$$
where $\epsilon$ is a Gaussian process, $X(s)$ is some rough covariate, and $S$ is some smoothing operator.  Examples of situations like this, where the observed data are smoother than the covariate are not uncommon. One example, which we will study in Section~\ref{sec:app} is standardized temperature and precipitation residuals over the midwest region of the United States (Figure~\ref{fig:data}). If we model the temperature residuals using the precipitation residuals as a covariate, we visually see that the observations are smoother than the covariate, and that there is a clear negative correlation between the observations and the covariate. 

\begin{figure}[t]
	\centering
	\includegraphics[width=0.8\textwidth]{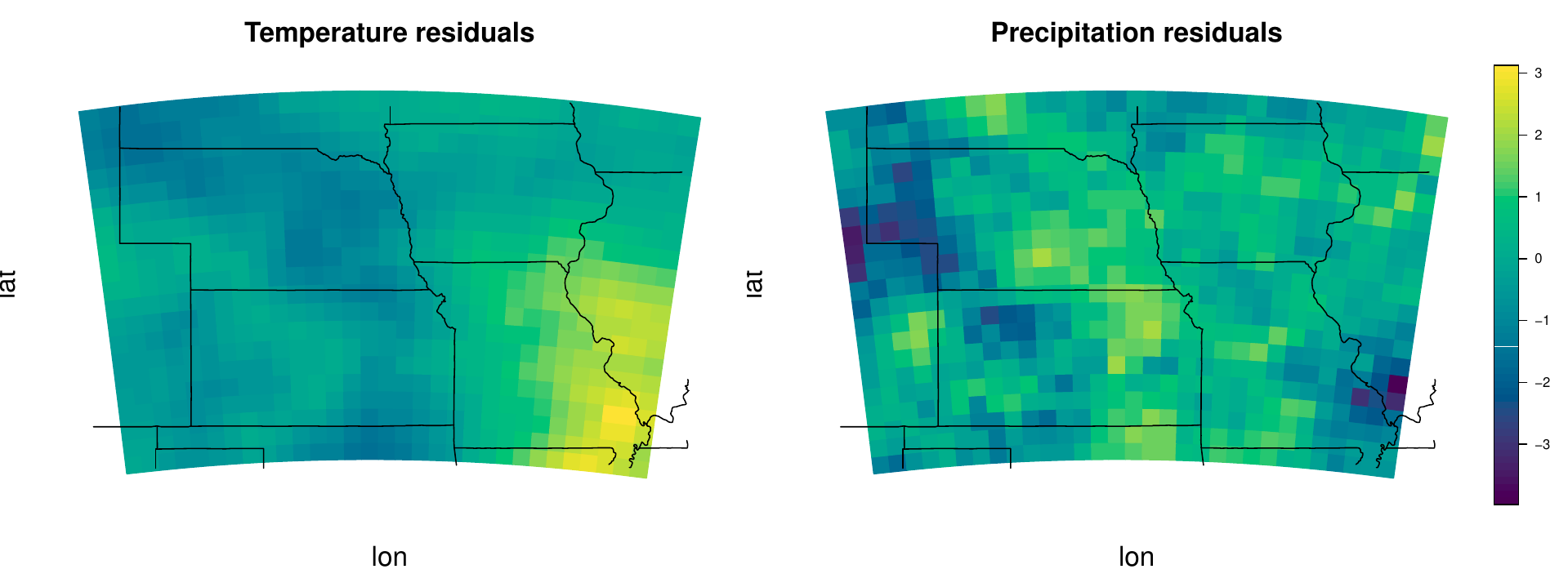}
	\caption{Standardized precipitation and temperature residuals.}
	\label{fig:data}
\end{figure}

Because only $Y(s)$ and $X(s)$ are observed, suppose that we use $X(s)$ as a covariate for $Y(s)$ in a Gaussian process regression
$$
Y(s) = X(s)\beta + \epsilon(s).
$$
In this case, we will show the somewhat counter intuitive result that $\beta\rightarrow 0$ as we get more data, no matter how high the correlation between $Y$ and $X$ is. Thus, using this model will severely underestimate the effect that $X(s)$ has on $Y(s)$.  Not surprisingly, a simple solution to this issue is to perform smoothing of the covariate before using it in the spatial regression.  However, we will show that one needs to do this with some care since, surprisingly, one can obtain that $|\beta|\rightarrow \infty$ as more data is obtained unless the covariate is smoothed enough.

In the next section, we formalize these problems and provide the main theoretical results, which hold for Gaussian process regression on general compact metric spaces. Thus, our theoretical results are much more general than the common setting of having observations in a bounded domain in $\mathbb{R}^d$, and are for example also applicable to data on the sphere or other manifolds. We also provide more detailed results for Gaussian random fields with fractional-order covariance operators, where Gaussian Whittle--Mat\'ern fields \citep{lindgren11, lindgren2022spde} are important special cases, and extend some of the results to observations with measurement noise. 
Two simulation studies are performed in Section~\ref{sec:sim} to illustrate the results, and Section~\ref{sec:app} contains the application to the temperature and precipitation residuals. We end the article with a discussion in Section~\ref{sec:discussion}, where we  relate these results to other recent results for spatial confounding. All proofs are given in the appendix of the manuscript.

\section{Theoretical setup and main results}
In this section, we provide the theoretical setup and the main results. We begin with the most general results and  
then focus on Gaussian fields with fractional-order covariance operators as a special case, which contains several important examples for which we can obtain refined results. 

\subsection{General results}
We let $(\cX,d_{\cX})$  denote a connected, compact metric space 
of infinite cardinality and let $\nu_\cX$ be a strictly positive and finite 
Borel measure on $(\cX,\cB(\cX))$, where $\cB(\cX)$ denotes the Borel $\sigma$-algebra on $\cX$.  We further assume that  $\epsilon$ is a 
square-integrable and centered Gaussian process 
$\epsilon \from\cX\times\Omega\to\bbR$ defined on a complete probability space $(\Omega,\cA,\bbP)$.
The covariance function of $\epsilon$,
$\varrho \from\cX\times\cX \to \bbR$, is assumed to be (strictly) positive definite and we denote the corresponding covariance operator by $\cC$:
\begin{equation}\label{eq:kernel-C}
	%\textstyle 
	\cC\from L_2(\cX,\nu_\cX) \to L_2(\cX,\nu_\cX), 
	\quad 
	(\cC w)(s) := \int_{\cX} \varrho(s, s') w(s') \,\rd\nu_\cX(s'), 
\end{equation}
Here, $L_2(\cX,\nu_\cX)$ is the Hilbert space of real-valued, Borel measurable, square-integrable functions on $\cX$, with inner product $(f, g)_{L_2(\cX,\nu_\cX)} = \int_{\cX} f(s) g(s) d\nu_{\cX}(s)$ and induced norm $\|f\|_{L_2(\cX,\nu_\cX)}^2 = (f, f)_{L_2(\cX,\nu_\cX)} = \int_{\cX}|f(s)|^2d\nu_{\cX}(s)$.

The Gaussian process induces a Gaussian measure $\mu = \pN(0,\cC)$
on the Hilbert space $L_2(\cX,\nu_\cX)$, i.e., 
for every Borel set $A$ we have 
$
\mu(A)
= 
\bbP(\{\omega\in\Omega : Y(\,\cdot\,, \omega)\in A\})$.
The Cameron--Martin space $\cC^{1/2}(L_2(\cX,\nu_\cX))$ associated with 
$\mu$ on $L_2(\cX,\nu_\cX)$, also known as the reproducing kernel Hilbert space of $\mu$,
is a Hilbert space which is defined as 
the range 
of~$\mathcal{C}^{1/2}$ in~$L_2(\cX,\nu_\cX)$ 
equipped with the inner product
$(\,\cdot\,, \,\cdot\,)_{\cC} := (\cC^{-1/2}\,\cdot\,, \cC^{-1/2}\,\cdot\,)_{L_2(\cX,\nu_\cX)}$ and induced norm $\|\cdot\|_\cC = \|\cC^{-1} \cdot\|_{L_2(\cX,\nu_\cX)}$. We similarly define $\cC^{1}(L_2(\cX,\nu_\cX))$ as the 
range 
of~$\mathcal{C}$ in~$L_2(\cX,\nu_\cX)$ 
equipped with the inner product
$(\,\cdot\,, \,\cdot\,)_{\cC^2} := (\cC^{-1}\,\cdot\,, \cC^{-1}\,\cdot\,)_{L_2(\cX,\nu_\cX)}$ and corresponding induced norm.

We are now interested in the maximum-likelihood estimation  of a model with a misspecified mean value. To make the results more generally applicable than to the setup mentioned in the introduction, we make the following assumption.

\begin{Assumption}\label{ass:1}
Suppose that $\{s_i\}_{i\in\bbN}$ is a sequence of locations in $\cX$ which is dense, i.e., each $s\in\cX$ is an accumulation point of the sequence, and let $\{Y_i\}_{i\in\bbN}$ be a corresponding sequence of observations, $Y_i = Y(s_i)$, of the process $Y\sim \mu_{\beta}$, where  
$
\mu_{\beta}=\pN(m+\beta \cS X,\cC).
$
Here $\cS : L_2(\cX,\nu_\cX) \rightarrow L_2(\cX,\nu_\cX)$ is some operator, $m\in \cC^{1/2}(L_2(\cX,\nu_\cX))$, and the covariate $X\neq 0$ satisfies $X \in L_2(\cX,\nu_\cX)$, $\|X\|_{\infty} < \infty$ and $\|\cS X\|_{\infty} < \infty$. 
\end{Assumption}

Our main theorem is the following. 

\begin{Theorem}\label{thm1}
	Suppose that the data is generated according to Assumption~\ref{ass:1} while our candidate model is that $Y\sim 	\hat{\mu}_{\beta}$, where $\hat{\mu}_{\beta}=\pN(\beta  X,\cC)$. Let $\hat{\beta}_n$ denote the maximum likelihood estimator (under the candidate model) based on the $n$ first observations, then:
	\begin{enumerate}[(i)]
		\item if $\cS X\notin \cC^{1/2}(L_2(\cX,\nu_\cX))$ and $(\cS X- X) \in \cC^{1/2}(L_2(\cX,\nu_\cX))$, then $\hat{\beta}_n \rightarrow \beta$ $\bbP$-a.s.;
		\item if $\cS X\in \cC^{1/2}(L_2(\cX,\nu_\cX))$ and $(\cS X- X) \notin \cC^{1/2}(L_2(\cX,\nu_\cX))$ then $\hat{\beta}_n \rightarrow 0$ $\bbP$-a.s.;
		\item if $ X\in \cC^{1}(L_2(\cX,\nu_\cX))$, then $\hat{\beta}_n \rightarrow  \beta_{\infty} = (Y,  X)_{\cC}/\| X\|_{\cC}^2$ $\bbP$-a.s., where $\beta_{\infty}$ is a random variable with $\pE(\beta_\infty) =  (m + \beta \cS X,  X)_{\cC}/\|  X\|_{\cC}^2<\infty$.
	\end{enumerate}
\end{Theorem}

The first two cases of this result can  be extended to the situation  where measurement noise is included in the observation equation. That is, suppose that the observations are 
\begin{align}
	\label{eq:measerr}
	\tilde Y_i = Y(s_i) + \epsilon_i,
\end{align} 
where $\epsilon_i \sim \mathcal{N}\left(0,\sigma^2_e\right)$, and $Y(s_i)$ satisfy Assumption~\ref{ass:1}. We then have the following result.
\begin{Corollary}\label{cor:measerr}
	Suppose that $\tilde Y_i$ follow \eqref{eq:measerr} while our candidate model is that $Y_i = Y(s_i) + \epsilon_i,$, where $Y \sim \hat{\mu}_{\beta}$ with $\hat{\mu}_{\beta} = \pN(\beta X, \cC)$. Let $\hat{\beta}_n$ denote the maximum likelihood estimator (under the candidate model) based on the $n$ first observations, then:
	\begin{enumerate}[(i)]
		\item if $\cS X\notin \cC^{1/2}(L_2(\cX,\nu_\cX))$ and $(\cS X- X) \in \cC^{1/2}(L_2(\cX,\nu_\cX))$ then $\hat{\beta}_n \rightarrow \beta$ $\bbP$-a.s.;
		\item if $\cS X\in \cC^{1/2}(L_2(\cX,\nu_\cX))$ and $(\cS X- X) \notin \cC^{1/2}(L_2(\cX,\nu_\cX))$ then $\hat{\beta}_n \rightarrow 0$ $\bbP$-a.s..
		\end{enumerate}
\end{Corollary}
To understand the implications of these results, let us explore a few important choices for the operator $\cS$ in the following corollaries. 

\begin{Corollary}
If $X= \cS X +X_0$ for an additive perturbation satisfying $X_0 \notin  \cC^{1/2}(L_2(\cX,\nu_\cX)),$ then $\hat{\beta}_n \rightarrow 0 $. 
\end{Corollary}

The following corollary provides the setup mentioned in the introduction. 
\begin{Corollary}
	If $\cS$ is a smoothing operator such that $\cS X \in \cC^{1/2}(L_2(\cX,\nu_\cX))$,  then $\hat{\beta}_n \rightarrow 0$ if $X \notin \cC^{1/2}(L_2(\cX,\nu_\cX))$.
\end{Corollary}

From these results, we see that we obtain $\hat{\beta}_n \rightarrow 0$ if $X$ is not in the Cameron--Martin space of the covariance operator, whereas $\cS X$ is. Thus, to avoid severe bias in these spatial regression problems, it is 
crucial to know whether or not the covariate, or the smoothed covariate, is in the Cameron--Martin space of the covariance operator. In the following subsection, we provide several examples of these spaces for Gaussian Whittle--Mat\'ern fields.

\subsection{Fractional-order covariance operators}\label{sec:fractional}
In this section, we consider an application of the results of the previous section to Gaussian random fields with fractional order covariance operators, and provide some additional results not covered in the general case. 
For this, we need some additional notation. 
Suppose that $A : \cD(A) \subsetneq L_2(\cX, \nu_\cX) \rightarrow L_2(\cX, \nu_\cX)$ is a densely defined, self-adjoint and positive linear operator with a compact inverse \citep[see][for details]{bk-measure}. 
In this case, $A$ has a set of eigenvectors $\{e_j\}_{j\in\bbN}$ and corresponding positive eigenvalues $\{\lambda_j\}_{j\in\bbN}$ which can be taken in a non-decreasing order and which only accumulate at $\infty$. We suppose that the there exists constants $c,C>0$ and $\eta >0$ such that 
\begin{equation}\label{eq:weyl}
c j^\eta \leq \lambda_j \leq C j^{\eta}, \quad \forall j\in\bbN.
\end{equation}
For $\alpha>0$, the action of the fractional-order operator 
${A^\alpha :  \cD(A^\alpha) \subsetneq L_2(\cX, \nu_\cX) \rightarrow L_2(\cX, \nu_\cX)}$ 
is defined in the spectral sense as
$$
A^\alpha f := \sum_{j=1}^\infty \lambda_j^\alpha (f, e_j)_{L_2(\cX, \nu_\cX)} e_j,
$$
for $f\in \cD(A^\alpha)$. For $s>0$, the domain of $A^s$ is defined as $\dot{H}_A^{2s} := \cD(A^{s})$, where
$$
\cD(A^s) =  \{ f \in L_2(\cX, \nu_\cX) : \|A^{s} f\|_{L_2(\cX, \nu_\cX)}^2 = \sum_{j=1}^\infty \lambda_j^{2s} (f, e_j)_{L_2(\cX, \nu_\cX)}^2 < \infty\}.
$$
We let $\dot{H}_A^{-s}$ denote the dual space of $\dot{H}_A^{s}$ and note that 
for $s>0$, $\dot{H}_A^s$ is a separable Hilbert space with inner product $(f,g)_{s,A} := (A^{s/2}f, A^{s/2}g)_{L_2(\cX, \nu_\cX)}$.  Now, suppose that the covariance operator $\cC$ of the Gaussian process is $\cC = A^{-\alpha}$ for some $\alpha>0$, then the Cameron--Martin space is $C^{1/2}(L_2(\cX, \nu_\cX)) = \dot{H}_A^{\alpha}$. This allows us to quantify the requirements regarding the Cameron--Martin space through $\dot{H}_A^{\alpha}$, and we now provide some explicit examples of these for the popular Whittle--Mat\'ern fields and generalized Whittle--Mat\'ern fields on various spatial domains. 

\begin{example}\label{ex:Rd}
	Suppose that $\cX = D \subset \bbR^d$ is connected, bounded and open domain with Lipschitz boundary. Choosing $A = \kappa^2 - \nabla\cdot(a\nabla)$ for a bounded function $\kappa : \overline{D} \rightarrow \bbR$ and a symmetric Lipschitz and uniformly positive definite function $a : \overline{D} \rightarrow \bbR^d$, and endowing the operator with homogeneous Dirichlet or Neumann boundary conditions, results in a generalized Whittle--Mat\'ern field \citep[see, e.g.][]{BK2020rational}. With $a$ as the identity function and $\kappa$ as a positive constant, we obtain the standard Whittle--Mat\'ern fields as a special case. For these fields, one can characterize the Cameron--Martin spaces in terms of standard fractional Sobolev spaces $H^{\alpha}(\cX)$, see, e.g., \cite{bk-measure}. By these results, $f\in H^{\alpha}(D)$ is required to have $f\in \dot{H}^{\alpha}_A$, and we therefore have an explicit smoothness requirement on the functions in the Cameron--Martin space.
\end{example}	
	
\begin{example}
	Suppose that $\cX = M$ is a closed, connected, orientable, smooth and compact surface in $\bbR^3$. Choosing $A = \kappa^2 - \nabla_M\cdot(a\nabla_M)\Delta_{\cX}$, where $\nabla_M\cdot$ is the surface divergence and $\nabla_M$ the surface gradient, results in a generalized Whittle--Mat\'ern field on the surface. The characterization of the Cameron--Martin space can then be done in terms of standard fractional Sobolev spaces $H^\alpha(M)$ \citep{Herrmann2020}. These results can also be extended to more general Riemannian manifolds, and in particular are applicable to Gaussian Whittle--Mat\'ern fields on manifolds by considering $A = \kappa^2 - \Delta_M$, where  $\kappa>0$ is constant and $\Delta_M$ is the  Laplace--Beltrami operator \citep{borovitskiy2020mat,harbrecht2021multilevel}. 
\end{example}

\begin{example}
	Suppose that $\cX=\Gamma$ is a compact metric graph, such as a linear network. Choosing $A = \kappa^2 - \Delta_\Gamma$, where $\kappa>0$ and $\Delta_\Gamma$ is the Kirchhoff-Laplacian results in a Whittle--Mat\'ern field on $\Gamma$ for which \citet{BSW2022} provide a representation of the Cameron--Martin spaces. \citet{bolinetal_fem_graph} extends this to generalized Whittle--Mat\'ern fields on $\Gamma$, obtained by considering $A = \kappa^2 - \md(a\md)$, where $\kappa$ is a bounded function, $a$ is a positive Lipschitz function, and $\md$ is an operator acting as the derivative on the edges of $\Gamma$. 
\end{example}

We assume that the operator $\cS$ is of the form $\cS = A_S^{\gamma}$, where $A_S$ is an operator that diagonalizes with respect to the eigenvectors $\{e_j\}_{j\in\bbN}$, with corresponding positive eigenvalues $\{\lambda_{S,j}\}$ satisfying the same asymptotic growth as $A$:
$$
c_s j^{\eta} \leq \lambda_{S,j} \leq C_S j^{\eta}, \quad \forall j\in\bbN,
$$
for some constants $c_S,C_s>0$. Then, from Theorem \ref{thm1} it follows that:
\begin{Corollary}\label{cor:betawm}
	Suppose that the data is generated according to Assumption \ref{ass:1}, where $\cC = A^{-\alpha}$ for $\alpha>0$ and that $\cS = A_S^{\gamma}$ for $\gamma\in\mathbb{R}$. Further suppose that $X \in \dot{H}_A^{p}$, for some $p>0$, and that $X \notin \dot{H}_A^{\tilde{p}}$ for any $\tilde{p}>p$. The candidate model is $Y \sim \hat{\mu}_{\beta}$, where $\hat{\mu}_{\beta}=\pN(\beta  X,\cC)$. Let $\hat{\beta}_n$ denote the maximum likelihood estimator  based on the $n$ first observations, then:
	\begin{equation}\label{eq:beta_conv}
	\hat{\beta}_n  \rightarrow \begin{cases}
		\beta & \mbox{ if }  p < \alpha \text{ and } \gamma = 0,  \\
		0 & \mbox{ if }  p < \alpha \text{ and } 2\gamma \leq p-\alpha,  \\
		\beta_{\infty} = (Y,  X)_{\cC}/\| X\|_{\cC}^2 &  \text{if $p \geq 2\alpha$ and $  \gamma < 0$},
	\end{cases}
	\end{equation}
	$\bbP$-a.s., where $\beta_{\infty}$ is a random variable with $\pE(\beta_{\infty}) = \beta(\cS X,X)_{\cC}/\|X\|_{\cC}^2<\infty$.
	\end{Corollary}

We can note that some combinations of $p, \alpha$ and $\gamma$ are missing in the result. It is more complicated to get a general result in these cases as the behavior appear to depend on how the observations are taken. One situation where we can obtain a complete characterization is if the observations are eigenbasis observations, which is a common assumption is several theoretical investigations of maximum likelihood estimators, \citep[see, e.g.,][]{cialenco2018}.

	\begin{Proposition}\label{The:betan2}
		Suppose that $\cC = A^{-\alpha}$ for $\alpha>0$ and that $\cS = A_S^{\gamma}$ for $\gamma\in\mathbb{R}$. Further suppose that $X \in \dot{H}_A^{p}$ for $p>0$, and that $X \notin \dot{H}_A^{\widetilde{p}}$ for $\widetilde{p}>p$.  The candidate model is $Y \sim \hat{\mu}_{\beta}$, where $\hat{\mu}_{\beta}=\pN(\beta  X,\cC)$. Let $\hat{\beta}_n$ denote the maximum likelihood estimator  based on the $n$ observations 
		$Y_i= \left(Y,e_i\right)_{L_2(\cX, \nu_\cX)}$ for $i=1,\ldots,n$ then 
		$$
		\hat{\beta}_n  \rightarrow \begin{cases}
			\beta 	& \text{if $p <\alpha$ and $\gamma = 0$,} \\
			0 		& \text{if $p <\alpha$ and $\gamma < 0$,} \\
		sign(\beta)	\infty 	& \text{if ($p <\alpha$ and $\gamma > 0$) or ($p \geq \alpha$ and $\gamma > p - \alpha$),} \\
			\beta_{\infty} = (Y,  X)_{\cC}/\| X\|_{\cC}^2& \text{if $p \geq \alpha$ and $\gamma \leq p - \alpha$},
		\end{cases}
		$$
		$\bbP$-a.s., where $\beta_{\infty}$ is finite with $\pE(\beta_{\infty}) = \beta(\cS X,X)_{\cC}/\|X\|_{\cC}^2$.
	\end{Proposition}

This result should carry over to any reasonable fill in limit sequence of observations such that $\mv{X}^{\top}\mv{\Sigma}^{-1}\mv{Y}$ and $\mv{X}^{\top}\mv{\Sigma}^{-1}\mv{X}$ in \eqref{eq:gls}  converges to $(Y, X)_{\cC}<\infty$ and $\| X\|_{\cC}^2<\infty$ as $n\rightarrow\infty$.
	Thus, we conjecture that for most practical situations, the results obtained for eigenbasis observations hold also for the point observations, in particular that $|\hat{\beta}_n| \rightarrow \infty$ if 
	$p <\alpha$ and $\gamma > 0$ or if $p \geq \alpha$ and $\gamma > p - \alpha$. This is also what we observe in Section~\ref{sec:simstudy2} for point observations. 

\section{Simulation experiments}
\label{sec:sim}
To illustrate the results of the previous section, we here perform a few simulation experiments. In the first experiment we take a time series approach and use a lowess smoother \citep{cleveland1979} to illustrate the effect of the results. In the second experiment, we consider a setup similar to the spatial data application in Section~\ref{sec:app}.

\subsection{A time series example}

 Suppose that we are given data 
$
Y_i = \cS X(s_i) + Z(s_i)
$
where $s_i \in [0,10]$, and $Z(s)$ is a centered Gaussian process with a Mat\'ern covariance with $\kappa=\nu=1$ and $\sigma=0.1$. We generate $\cS X$ by first simulating a realization $ X(s)$ of a centered Gaussian process with a Mat\'ern covariance with $\kappa=\nu=1$ and $\sigma=0.4$, and then smoothing it with a lowess smoother with smoother span $0.1$. An example of the process and covariates can be seen in Figure~\ref{fig:ex1}.

\begin{figure}[t]
	\includegraphics[width=\linewidth]{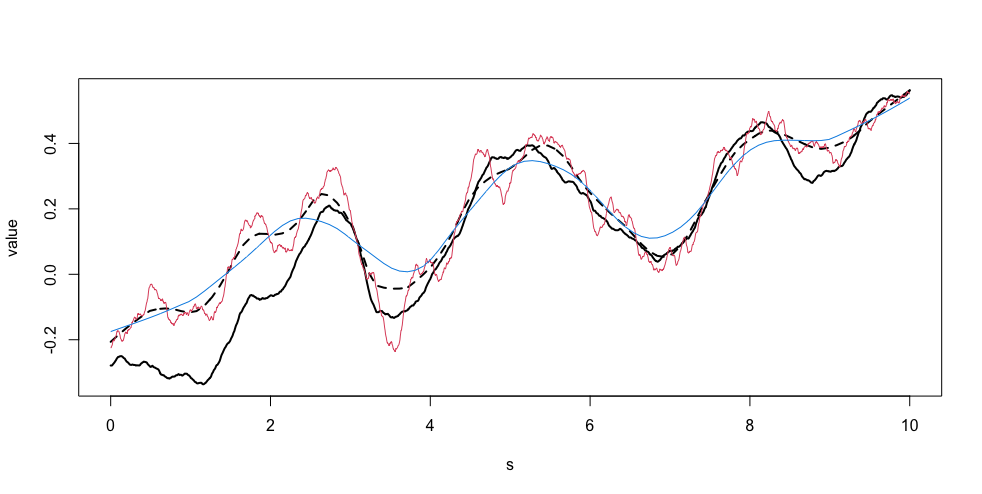}
	\caption{Simulation of $Y(s) = \cS X(s) + Z(s)$ (black solid), $\cS X(s)$ (black dashed), the non-smoothed covariate $X(s)$ (red), and the smoothed covariate $\hat{\cS} X(s)$ (blue).}
	\label{fig:ex1}
\end{figure}

As a first experiment, we assume that we know the covariance function of $Z$ and estimate two models 
\begin{equation*}
\text{Model 1} : Y_i = \beta  X(s_i) + Z(s_i), \qquad
\text{Model 2} : Y_i = \beta  \hat{\cS} X(s_i) + Z(s_i),
\end{equation*}
where the first model uses the unmodified, rough, covariate $ X$, and  the second model uses $\hat{\cS} X$ which is a smoothed version of $X$, obtained by a different lowess smoother with smoother span $0.2$. 

For both models, we estimate $\beta$ using maximum likelihood using an increasing  amount of observations, regularly spaced in the interval $[0,1]$. The resulting estimates can be seen in the columns for Experiment 1 in Table~\ref{tab1}, where the values are obtained as averages of the results for $10$ repetitions of the experiment. 

\begin{table}[t]
	\centering
	\begin{tabular}{@{}lcc cc cc@{}}
		& \multicolumn{2}{c}{Experiment 1} & \multicolumn{2}{c}{Experiment 2} & \multicolumn{2}{c}{Experiment 3} \\
		\cmidrule(lr){2-3}  \cmidrule(lr){4-5}  \cmidrule(l){6-7}
		$n $ & Model 1 & Model 2 & Model 1 & Model 2  & Model 1 & Model 2 \\
		\hline
		$128$ & $0.108$ & $1.059$ & $0.027$ & $1.043$ & $0.026$ & $1.006$\\ 
		$256$ & $0.052$ & $1.049$ & $0.018$ & $1.081$ & $0.042$ & $1.023$\\ 
		$512$ & $0.026$ & $1.046$ & $0.011$ & $1.094$ & $0.034$ & $1.04$\\ 
		$1024$ & $0.014$ & $1.048$ & $0.009$ & $1.097$ & $0.034$ & $1.039$\\ 
		\hline
	\end{tabular}
\caption{Estimates of the regression coefficient $\beta$ for the three simulation experiments. Model 1 has a rough covariate whereas Model 2 has a smooth covariate. Each number is an average over $10$ different datasets. }
\label{tab1}
\end{table}

As a second experiment, we redo the same simulations, but estimate the parameters $\kappa,\sigma,\nu$ of the covariance function jointly with $\beta$ when performing the maximum likelihood estimation. The results are shown in the columns for Experiment 2 in Table~\ref{tab1}.
Finally, as a third experiment, we extend the model for $Z$ by adding a nugget effect. That is, we use a covariance function $r(h) = \rho(h) + \sigma_e^2 1(h=0)$, where $\sigma_e^2$ is the nugget representing the variance of measurement error. We generate the data $Y_i$ assuming this covariance for $Z$, where we choose $\sigma_e = 0.01$, and estimate all parameters $\kappa,\sigma,\nu,\sigma_e$ of the covariance function jointly with $\beta$ when performing the maximum likelihood estimation. The results of this experiment are shown in the columns for Experiment 3 in Table~\ref{tab1}.

In all three experiments, the estimate of $\beta$ tends to zero as the number of observations increases for the first model, whereas the estimate is stable for the second model, despite that it uses the wrong covariate. This result is expected from the results of the previous section, but might seem  counterintuitive given that $X$ has a high correlation with the observed data (it was, for example, $0.951$ on average for the first experiment).

\subsection{A spatial example}\label{sec:simstudy2}
We now illustrate the results in an example related to the application in the next section. 
We assume that $\epsilon$ in \eqref{eq:main_model} is a Gaussian Mat\'ern field with range $\kappa=0.4$, variance $\sigma^2=1.3^2$, and smoothness $\nu=2$ on the domain in Figure~\ref{fig:data}. These parameter values were chosen so that they are similar to parameters for the application. 
We further assume that $X$ also is a Gaussian Mat\'ern field, independent of $\epsilon$, with the same parameters $\kappa$ and $\sigma$, and a different smoothness parameter $\nu_X$. Finally, we assume that $\cS$ is chosen so that $\cS X$ also is a Gaussian Mat\'ern field with the same parameters $\kappa$ and $\sigma$ and a different smoothness parameter $\nu_{SX}$.   

We consider 620 observation locations as in the Figure \ref{fig:data} for the application.  If we let $\mv{Y}$ denote the observations at these locations, the model specifically is
\begin{align*}
	\mv{Y} &=  \cS\mv{X} +   \mv{\Sigma}_{\epsilon}^{1/2} \mv{Z}_1,
\end{align*}
where $\cS\mv{X} =   \mv{\Sigma}_{\cS X}^{1/2}\ \mv{Z}_2$ and 
$\mv{X} =  \mv{\Sigma}_X^{1/2}\ \mv{Z}_2$. Here, $\mv{Z}_i$ are vectors with independent standard Gaussian variables and $\mv{\Sigma}_{\epsilon},\mv{\Sigma}_X$ and $,\mv{\Sigma}_{\cS X}$ denote the covariance matrices corresponding to Mat\'ern covariance functions with parameters $\kappa = 0.4$,  $\sigma = 1.3$ and smoothness $\nu = 2,\nu_X,$ and $\nu_{\cS X}$ respectively. 
For $\nu_{\cS X}$ we consider the two cases $\nu_{\cS X}=\nu$ and $\nu_{\cS X}=\nu-0.5$. In the first, the smoothness of $\mv{\epsilon}$ and $\cS\mv{X} $ are equal which is likely the situation one is in when estimating the smoothness from data. In the second case, $\cS\mv{X}$ is rougher than $\mv{\epsilon}$, which could occur if the smoothness of $\mv{\epsilon}$ is kept fixed (e.g., if an exponential or a squared exponential covariance is used).
  
 We subsample $\mv{Y}$, $\mv{X}$ and $\cS\mv{X}$ by sampling $n$ out of the 620 observation locations uniformly at random and compute $\hat{\beta}_n$ using \eqref{eq:gls} where $n$ is the number of observations in the sample. These estimates are computed for different values of $\nu_X \in [0.5,10]$. 
 
 In the notation of Section~\ref{sec:fractional}, the smoothness parameters correspond to 
 $\alpha = \nu+1$, 
 ${p = \nu_X}$ and 
 $\gamma = (\nu_X - \nu_{SX})/2$. Thus, 
 if Proposition~\ref{The:betan2} also would hold for point observations,
 we expect that
 $$
\text{case 1: } \hat{\beta}_n\rightarrow 
 \begin{cases}
0 & \text{if $\nu_X<2$,}\\
\beta & \text{if $\nu_X=2$,}\\
\infty & \text{if $2< \nu_X< 3$,}\\
\beta_{\infty} & \text{if $\nu_X\geq 3$,}
 \end{cases}
 \quad \text{case 2: } \hat{\beta}_n\rightarrow 
 \begin{cases}
0 & \text{if $\nu_X<1.5$,}\\
\beta & \text{if $\nu_X=1.5$,}\\
\infty & \text{if $1.5< \nu_X< 3$,}\\
\beta_{\infty} & \text{if $\nu_X\geq 3$,}
 \end{cases}
 $$
with $\pE(\beta_{\infty}) = \beta(\cS X, X)_{\cC}/\|X\|_{\cC}<\infty$. 

\begin{figure}[t]
	\centering
	\includegraphics[width=0.45\textwidth]{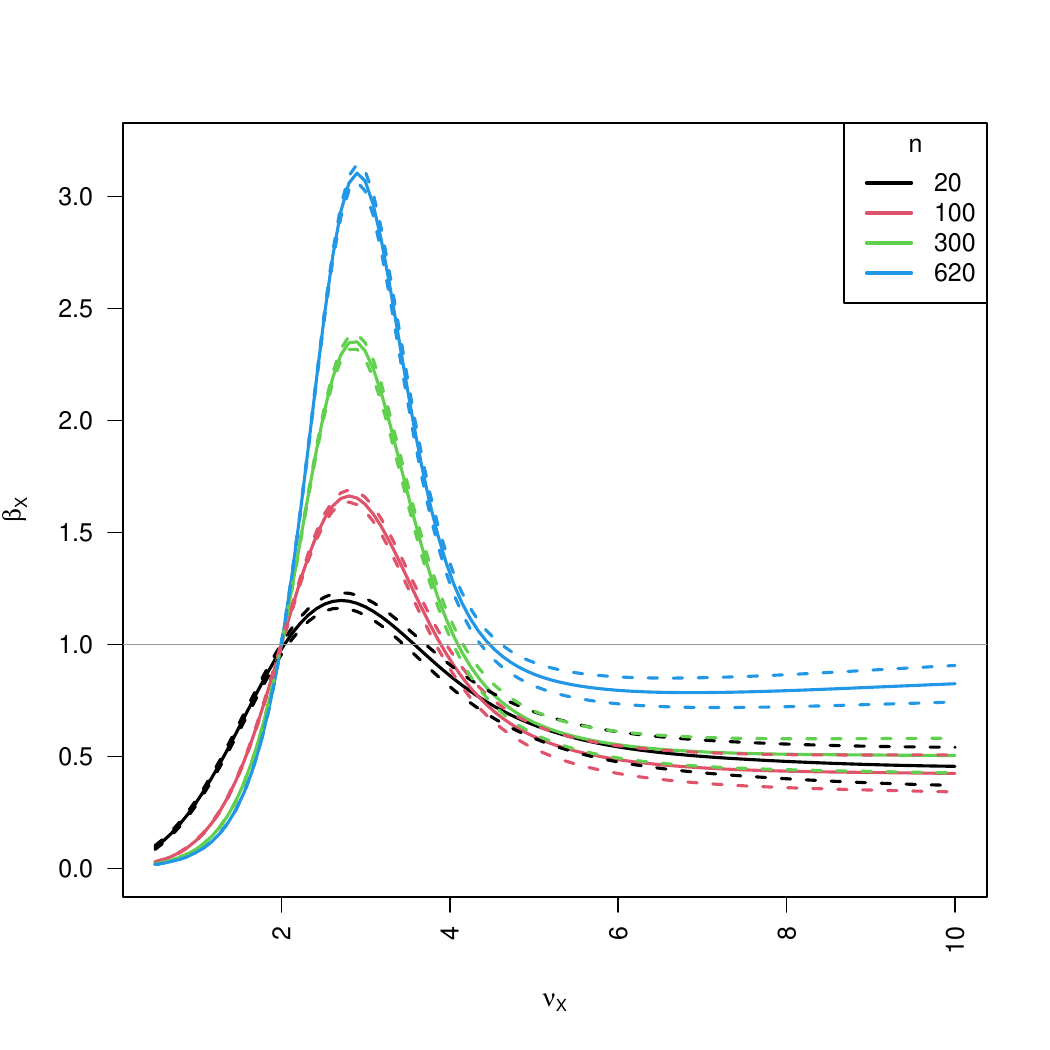}
	\includegraphics[width=0.45\textwidth]{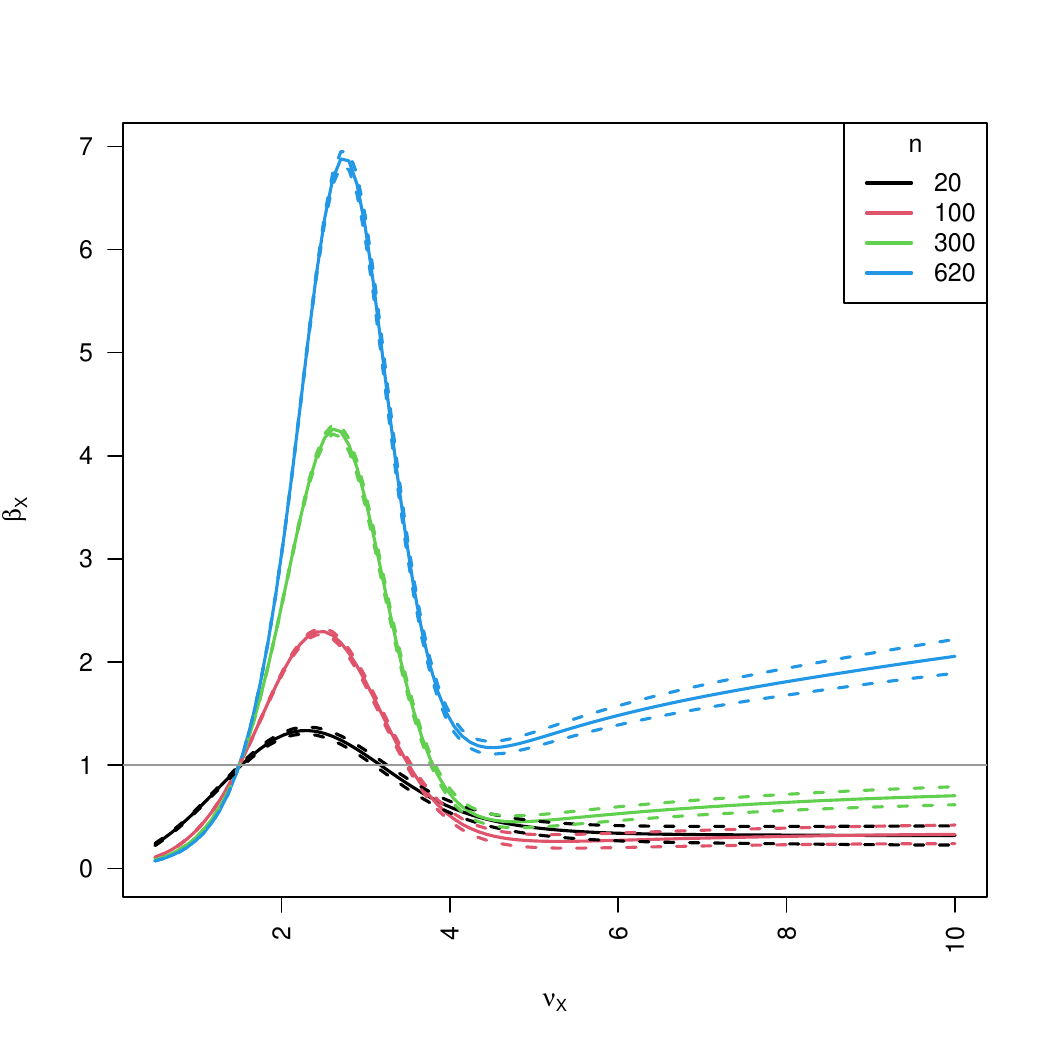}
	\caption{The effect of smoothing the covariate for the cases $\nu_{\cS X}=\nu$ (left) and $\nu_{\cS X}=\nu-0.5$ (right). The solid line is the mean effect of $400$ simulations, and the dashed lines represent  Monte Carlo $95\%$ confidence bands. }
	\label{fig:nuX}
\end{figure}
 We repeat the experiment 400 times, and the result is seen in Figure \ref{fig:nuX}. One can observe the result from Proposition~\ref{The:betan2} emerging as the sample size grows.  In particular when $\nu_X < 2$ for the first case, and $\nu_X < 1.5$ for the second, the estimator goes to zero. When $\nu_X \in (2,3)$ for the first case, and $\nu_X \in (1.5,3)$ in the second, there is a positive bias. As the sample size is not that large in this example, it is not clear if this bias will be unbounded if $n$ was increased further, but the bias is clearly large enough that it could have serious implications for the interpretation of the results.

\section{An application to reanalysis data}
\label{sec:app}
To illustrate the theory on real data we take a data set of precipitation and temperature reanalysis. The variables are the average summer (June, July and August) temperature, $T(s,t)$, and the cube-root precipitation, $P(s,t)$, over the Midwest United states for 24 years (1981-2004). 
To simplify interpretation we standardize the data for each year.
The data for one year is shown in Figure \ref{fig:data}. The empirical correlation between the two variables is $-0.67$. This data was studied earlier in \cite{crosscov} for examining the joint distribution through various multivariate models, and they in particular studied a linear coregionalization model,
\begin{align*}
	T(s,t) &= a_{11} Z_1(s,t), \\
	P(s,t) &= a_{12} Z_1(s,t) + a_{22} Z_2(s,t),  
\end{align*}
where $Z_1,Z_2$ are independent Gaussian random fields. The formulation is motivated by ``We opt for this formulation since temperature is expected to be smoother than precipitation,...''. Therefore the data seems ideal to study in our context as differentiability of the field determines the Cameron--Martin space of the Mat\'ern covariance.

\begin{table}[t]
	\centering
	\begin{tabular}{@{}lcr cr @{}}
		& \multicolumn{2}{c}{Temperature model} & \multicolumn{2}{c}{Precipitation model} \\
		\cmidrule(lr){2-3}\cmidrule(lr){4-5}
		$n (space)$ & $\beta_P (CI95\%)$ & RMSE & $\beta_T(CI95\%)$ &  RMSE   \\
		\hline
		$15$ & $-0.21\,(-0.26,-0.16)$ & $0.19$ & $-0.55\,(-0.67,-0.44)$ & $0.41$   \\
		$50$ & $-0.18\,(-0.20,-0.16)$ & $0.18$ & $-0.73\,(-0.81,-0.65)$& $0.42$   \\
		$100$ & $-0.15\,(-0.16,-0.14)$ &$0.17$ & $-0.80\,(-0.86,-0.73)$& $0.42$   \\
		$250$ & $-0.11\,(-0.12,-0.11)$ &$0.17$ & $-1.07\,(-1.13,-1.01)$& $0.43$   \\
		$620$ & $-0.09\,(-0.09,-0.09)$& $0.17$ & $-1.57\,(-1.60,-1.50)$& $0.46$   \\
		\hline
	\end{tabular}
	\caption{Results for modelling temperature with precipitation as covariate and precipitation with temperature as covariate. The table shows the estimated regression coefficients, their $95\%$ confidence intervals, and the root mean square prediction error for the prediction set.}
	\label{tab:res}
\end{table}
In order to study how the estimators behave as functions of sample size, we subsample $n$ of the $620$ spatial locations, keeping the data from all $24$ years for the subsampled locations. 
We explore both predicting temperature using precipitation as a covariate,
$$
T(s,t) = \beta_0+ \beta_P P(s,t)+   Z_T(s,t;\theta_T),
$$
and predicting precipitation using temperature as a covariate,
$$
P(s,t) =\beta_0+ \beta_T T(s,t)+   Z_P(s,t;\theta_P).
$$
Here $Z_T$ and $Z_P$ are fitted using a  Mat\'ern covariance with parameters estimated from the subsampled data.  We estimate the regression coefficients and covariance parameters based on the subsampled dataset. Then, after fitting the parameters we split the data set into a prediction set of $36$ locations, which is predicted given the remaining $ 584$ observations. The result is shown in Table \ref{tab:res}. The coefficients change as the theory suggests, and one can note the absurd implication of using this model for predicting precipitation using temperature: As the data is standardized, the yearly variance of both pressure and temperature is one, while the predictor $\beta_T T(s,t)$ has a variance of $1.57^2$. Also note that the prediction error is slightly decreasing with $n$ for temperature, even though the coefficient goes to zero, this is inline with theory as the unsmoothed precipitation cannot contribute to the interpolation.

Finally, we fit the same models but using smoothed covariates. These smoothed covariates are obtained by first estimating the covariance of the covariates, i.e., we estimate the models $P(s,t)=\beta_0+Z_P(s,t;\theta^{fixed}_{P})$ and $T(s,t)=\beta_0+Z_T(s,t;\theta^{fixed}_{T})$. Then we multiply the observed precipitation $\mv{P}$ and temperature $\mv{T}$ by some powers of the fitted covariance matrices, yielding $\cS\mv{P}= \mv{\Sigma}^3_{\theta^{fixed}_{P}}\mv{P}$  and $\cS\mv{T}= \mv{\Sigma}^{0.5}_{\theta^{fixed}_{T}}\mv{T}$. These choices are likely not optimal, but the results are sufficient to ensure that the covariates are smooth enough as can be seen in the stable estimates in Table \ref{tab:ressmooth}. Note also that there is no large difference in prediction error compared to the results in Table~\ref{tab:res}, which means that the problem cannot be identified through  cross validation.

\begin{table}[t]
	\centering
	\begin{tabular}{@{}lcr cr @{}}
		& \multicolumn{2}{c}{Temperature model} & \multicolumn{2}{c}{Precipitation model} \\
		\cmidrule(lr){2-3}\cmidrule(lr){4-5}
		$n (space)$ & $\beta_{SP} (CI95\%)$ & RMSE & $\beta_{ST}(CI95\%)$ &  RMSE   \\
		\hline
		$15$ & $-0.32\,(-0.44,-0.19)$ & $0.17$ & $-0.37\,(-0.50,-0.25)$ & $0.42$  \\
		$50$ & $-0.19\,(-0.33,-0.06)$ & $0.17$ & $-0.48\,(-0.59,-0.38)$& $0.42$   \\
		$100$ & $-0.19\,(-0.33,-0.04)$& $0.17$ &$-0.55\,(-0.65,-0.46)$& $0.42$   \\
		$250$ & $-0.31\,(-0.46,-0.17)$& $0.17$ &$-0.58\,(-0.68,-0.49)$& $0.42$  \\
		$620$ & $-0.36\,(-0.53,-0.19)$& $0.17$ &$-0.58\,(-0.67,-0.49)$& $0.42$   \\
				\hline
	\end{tabular}
	\caption{Results for modelling temperature with smoothed precipitation as covariate and precipitation with smoothed temperature as covariate. The table shows the estimated regression coefficients, their $95\%$ confidence intervals, and the root mean square prediction error for the prediction set.}
	\label{tab:ressmooth}
\end{table}

\section{Discussion}\label{sec:discussion}
We have investigated the implications of misspecifying the smoothness of covariates in regression models within an infill asymptotic framework, and in particular showed the importance of ensuring that the covariates do not exhibit rougher characteristics than the underlying field. If the covariates are too rough, the corresponding regression coefficients will converge to zero no matter how high the correlation is with the observed data. 
This main result is formally established in Theorem \ref{thm1}. 

On the other hand, if a misspecified covariate possess roughness equal to or slightly smoother than the field, the corresponding regression coefficient may diverge to infinity under certain scenarios. We were unable to establish this result solely based on the assumption of observations filling the domain; indeed, we suspect that the result need not be true under such weak conditions. Instead, we demonstrated the behavior utilizing observations of eigenvectors, as detailed in Proposition \ref{The:betan2}. However, both simulated data from random locations (see Section \ref{sec:simstudy2}) and real-world data (see Section \ref{sec:app}) indicate that this behavior holds for point observations at random locations, at least for stationary Gaussian random fields with Mat\'ern covariances.

We have not encountered any previous works studying misspecified covariates under infill asymptotics. There are, however, related works worth commenting on. For example in \cite{gilbert2023consistency}, the regression coefficient is defined as identifiable if it is less smooth than $\epsilon$, underscoring the significance of smoothness. While their definition of identifiability aligns with the ability to perfectly separate noise from regression coefficients, we prefer  to denote this as a lack of consistent estimability, to avoid confusion with the classical definition of identifiability (refer to Definition 5.2 in \cite{lehmann2006theory}).
Another pertinent work is \cite{dupont2023demystifying}, which conducts a non-asymptotic analysis of spatial confounding effects and highlights the impact of smoothing. On page 8, they state, ``the expression in Corollary 3.2 shows firstly that bias in the spatial model is directly linked to the smoothing in the model, as without smoothing the bias would not arise.''
Finally, a recent method proposed by \cite{dupont2022spatial} suggests removing the smooth component of $X$ using a thin plate spline. This recommendation appears contrary to our findings. However, the authors clarify in their discussion that their methods are designed for covariates with non-spatial components, implying a measurement error component within $X$. Consequently, their approach would be suitable in scenarios where $\cS X \notin \cC^{1/2}(L_2(\cX,\nu_\cX))$ in our notation.

An intriguing avenue for future research is to explore general assumptions on the observations under which the regression coefficients tend to infinity under misspecified smoothness. Additionally, investigating whether these properties persist when substituting the Gaussian field $\epsilon$ in \eqref{eq:main_model} with a more general L\'evy random field is an interesting topic for future research.

\section*{Acknowledgments}\label{sec:acknowledgment}
The authors thanks Peter Diggle and 
Claudio Fronterre for helpful discussions and suggestions of relevant articles to study. Further the authors would like to thank Alexandre Simas for discussion of the results and help with certain proofs.
\begin{appendix}
	\section{Collected proofs}\label{sec:proofs}
	Throughout the proofs we will use representations of Gaussian fields in terms of Gaussian white noise $\cW$ on $L_2(\cX,\nu_{\cX})$, which can be represented as a family of centered Gaussian random variables $\{\cW(h) : h\in L_2(\cX, \nu_{\cX})\}$ such that
	$$
	\forall f,g \in L_2(\cX, \nu_{\cX}),\quad \pE(\cW(f)\cW(g)) = (f,g)_{L_2(\cX, \nu_{\cX})}.
	$$
	A Gaussian random field $\epsilon$ with covariance operator $\cC$ can then be represented as $\epsilon = \cC^{1/2}\cW$.
	\subsection{General results}
	\begin{proof}[Proof of Theorem~\ref{thm1}]
		Let $f^{\beta}_n$ denote the density of $Y_1,\ldots,Y_n$ under $\mu_{\beta}$, and let $\hat{f}_n(\beta)$ denote the corresponding density under $\hat{\mu}_\beta$. By the Feldman--Hajek theorem  \citep[Theorem~2.25]{daPrato2014}, we have that  $\mu_\beta$ and $\hat{\mu}_{\hat{\beta}}$ are equivalent if 
		\begin{equation}\label{eq:meandiff}
		 m + \beta \cS X - \hat{\beta} X  \in  \cC^{1/2}(L_2(\cX,\nu_\cX)),
		\end{equation}
		and they are orthogonal otherwise. 
		
		We start by proving (i). 
		Since $\cS X \notin \cC^{1/2}(L_2(\cX,\nu_\cX))$ and $\cS X- X \in \cC^{1/2}(L_2(\cX,\nu_\cX))$, we have that $X \notin \cC^{1/2}(L_2(\cX,\nu_\cX))$ and	
		\eqref{eq:meandiff} therefore shows that $\hat{\mu}_{\hat{\beta}}$ is equivalent to $\mu_\beta$ if $\beta = \hat{\beta}$ and that the measures are orthogonal whenever $\hat{\beta}\neq \beta$. 
		Because of this, we by \citet[][Theorem 1, p.442]{Gikhman1974} and \citet[][Proposition 2.26]{daPrato2014} have that 
		\begin{equation}\label{eq:limit1}
			\lim_{n\rightarrow \infty}\log\frac{f^\beta_n}{\hat{f}_n(\hat{\beta})} 
			= 
			\begin{cases}
				-\infty & \hat{\beta} \neq \beta,\\
				\log(c_0) & \hat{\beta} = \beta,
			\end{cases}
		\end{equation}
		with probability 1, where $c_0>0$ is the Radon--Nikodym derivative 
		\begin{equation}\label{eq:deriv}
		\frac{ \md \mu_\beta}{\md \hat{\mu}_{\beta}}(y) = \exp\left((y-m - \beta\cS X, \beta \left(\cS X-X \right)+m)_{\cC} - \frac1{2}\|\beta \left(\cS X-X \right) +m\|_{\cC}^2\right),
		\end{equation}
	%	\begin{align}
	%	\frac{ \md \mu_\beta}{\md \hat{\mu}_{\hat{\beta}}}(y) &= \exp\left((y-\hat{\beta} X, m + \beta \cS X-\hat{\beta} X )_{\cC} - \frac1{2}\|m + \beta \cS X-\hat{\beta}X \|_{\cC}^2\right),
	%\end{align}
	for $\hat{\beta} = \beta$. 
		When $\hat{\beta} = \beta$, $\|m + \beta \cS X-\hat{\beta}X \|_{\cC}$ is finite because  $\cS X - X \in \cC^{1/2}(L_2(\cX,\nu_\cX))$ and $m \in \cC^{1/2}(L_2(\cX,\nu_\cX))$. Further, for $\hat{\beta} = \beta$ we have that $m + \beta \cS X-\hat{\beta}X \in \cC^{1/2}(L_2(\cX,\nu_\cX)$ so that we can write $m + \beta \cS X-\hat{\beta}X = C^{1/2} f$ for $f \in L_2(\cX,\nu_{\cX})$. Using this and that 
		$y = m + \beta \cS X + \epsilon$, where $\epsilon = \cC^{1/2}\cW$, 
		we obtain that 
		\begin{align*}
		(y-\hat{\beta} X,  m + \beta \cS X-&\hat{\beta}X )_{\cC} = (\epsilon, m + \beta \cS X- \hat{\beta} X )_{\cC} + \|m + \beta \cS X- \hat{\beta} X \|_{\cC} \\
		&=  (C^{1/2}\cW, C^{-1} (m + \beta \cS X- \hat{\beta} X) )_{L_2(\cX, \nu_{\cX})} + \|m + \beta \cS X- \hat{\beta} X \|_{\cC} \\
		&=  (C^{1/2}\cW, C^{-1/2} f )_{L_2(\cX, \nu_{\cX})} + \|m + \beta \cS X- \hat{\beta} X \|_{\cC} \\
		&=  \cW(f) + \|m + \beta \cS X- \hat{\beta} X \|_{\cC} < \infty \quad \text{$\bbP$-a.s}.
		\end{align*}
		Thus, $c_0 < \infty$ $\bbP$-a.s.
		To prove the claim we need show that $\hat{\beta}_n $ is not a divergent sequence. 
		It is enough to show that for a fixed $\epsilon>0$, there exists an $N$ such that $\rho_n(\hat{\beta}) = \log f_n - \log \hat{f}_n(\hat{\beta}) < \log(c_0) - 1$ for all $|\beta-\hat{\beta}|>\epsilon$ whenever $n>N$. 
		To that end, fix a $\beta^*$ with $|\beta^*-\beta|>\epsilon$ and note that by \eqref{eq:limit1}, there is an $N$ such that $\rho_n(\beta^*) < \log(c_0) - 1$. 
		Next, it is clear that $\hat{\beta} \mapsto \log \hat{f}_n(\hat{\beta})$ is strictly concave, and we therefore have that  $\rho_n(\hat{\beta}) < \log(c_0) - 1$ for all $\hat{\beta}$ with $|\beta-\hat{\beta}|>|\beta^*-\hat{\beta}|$. 
		This finishes the proof since $\epsilon$ was arbitrary.
		
		We now prove (ii). Since $\cS X \notin \cC^{1/2}(L_2(\cX,\nu_\cX))$ and $(\cS X- X) \notin \cC^{1/2}(L_2(\cX,\nu_\cX))$,  \eqref{eq:meandiff} shows that $\hat{\mu}_{\hat{\beta}}$ is orthogonal to $\mu_\beta$ whenever $\hat{\beta}\neq 0$. We have that 
		\begin{equation}
			\lim_{n\rightarrow \infty}\log\frac{f^{\beta}_n}{\hat{f}_n(\hat{\beta})} 
			= 
			\begin{cases}
				-\infty & \hat{\beta} \neq 0,\\
				\log(c_0) & \hat{\beta} = 0,
			\end{cases}
		\end{equation}
		with probability 1, where $c_0>0$ is the Radon--Nikodym derivative \eqref{eq:deriv}, which is finite because $\hat{\beta} = 0$. The remaining proof is now identical to that of (i).
		
		To prove (iii) note that $\lim_{n\rightarrow \infty}\log\frac{f^\beta_n}{\hat{f}_n(\beta) }$ will be equal (with probability 1) to the Radon--Nikodym derivative \eqref{eq:deriv}.
		The maximum likelihood estimator of $\hat{\beta}$ is the maximizer of this expression, which is given by 
		\begin{align*}
			\beta^* &= \argmin_{\hat{\beta}}\left\{ (Y, \hat{\beta}  X)_{\cC} - \frac1{2}\|\hat{\beta}  X\|_{\cC}^2\right\} = \frac{(Y,  X)_{\cC} }{\|  X\|_{\cC}^2}.
		\end{align*}
		Now will show that this a random variable with finite variance and expectation.
		Since $X \in \cC^{1/2}(L_2(\cX,\nu_\cX))$ there exists an $\widetilde{X}\in L_2(\cX,\nu_\cX)$ such that $ X = \cC^{1/2}\widetilde{X}$. 
		Furthermore, $Y$ can be expressed as $m + \beta \cS X + \cC^{1/2}\cW$ where $\cW$ is Gaussian white noise. Thus, 
		\begin{equation}\label{eq:YXproof}
		(Y,X)_{\cC} = (m+ \beta \cS X,X)_{\cC}   + (\cC^{1/2}\cW,X)_{\cC}.
		\end{equation}
		Here 
		$
		(\cC^{1/2}\cW, X)_{\cC} = (\cC^{1/2}\cW,\cC^{-1/2}\widetilde{X})_{L_2(\cX,\nu_\cX) } = \cW(\widetilde{X}) < \infty$, $\bbP$-a.s.,
		since $\widetilde{X}\in L_2(\cX,\nu_\cX)$ and $\cW$ is Gaussian white noise on $L_2(\cX,\nu_\cX)$. 
		Finally, since $X \in \cC^{1}(L_2(\cX,\nu_\cX))$ there exists $\bar{X}\in L_2(\cX,\nu_{\cX})$ such that $X=\cC \bar X $ and 
		$
		(m+ \beta \cS X,X)_{\cC}=(m+ \beta \cS X, \bar X)_{L_2(\cX,\nu_\cX)} < \infty. 
		$
	\end{proof}
	
	\begin{proof}[Proof of Corollary~\ref{cor:measerr}]
		This is an immediate consequence of the proof of Theorem \ref{thm1} and \citet[][Theorem 6 (Chapter 4)]{stein99}.
	\end{proof}

	\subsection{Fractional-order covariance operators}

	\begin{proof}[Proof of Corollary~\ref{cor:betawm}]
		Let $\dot{H}_S^{2\gamma}$ denote the domain of $A_S^{\gamma}$ and note that we have the equivalence $\dot{H}_A^0 \cong L_2(\cX,\nu_{\cX}) \cong \dot{H}_S^0$.
		By \citet[][Lemma 2.1]{BKK2020}, $A_S^{\gamma}$ can be extended to an isometric isomorphism $A_S^{\gamma} : \dot{H}_S^{s} \rightarrow \dot{H}_S^{s-2\gamma}$ for $s\in\mathbb{R}$. Further, because $\lambda_{S,j}$ and $\lambda_j$ have the same asymptotic growth, we have that $\dot{H}_A^s \cong \dot{H}_S^s$ for $s\in\mathbb{R}$. Thus, we have that $\cS X \in \dot{H}_A^{p-2\gamma}$.
			
		Using the same notation as in the proof of Theorem~\ref{thm1}, we have that $\mu_\beta$ and $\hat{\mu}_{\hat{\beta}}$ are equivalent if and only if $\beta \cS X - \hat{\beta} X \in \dot{H}_A^{\alpha}$.
		If $\gamma=0$ and $p<\alpha$, we have that $\cS X \notin \dot{H}_A^{\alpha}$ and $\cS X - X = 0 \in \dot{H}_A^{\alpha}$. Therefore, the result follows from Theorem~\ref{thm1} (i). 
		If $p < \alpha$ and $2\gamma < p - \alpha$, we have that $\cS X \in \dot{H}_A^{p - 2\gamma} \subseteq \dot H_A^{\alpha}$ and $\cS X - X \notin \dot{H}_A^{\alpha}$. Therefore, the result follows from Theorem~\ref{thm1} (ii) as $\gamma$ is negative so that $\cS : L_2(\cX,\nu_{\cX}) \rightarrow L_2(\cX,\nu_{\cX})$. 
		Finally, if $p\geq2\alpha$ and $  \gamma \leq 0$ , we have that $X\in \dot{H}_A^{2\alpha}$ and the result follows from Theorem~\ref{thm1} (iii).
	\end{proof}	

	\begin{proof}[Proof of Proposition~\ref{The:betan2}]
		When the observations are eigenfunction observations, the likelihood is given by
			$
			l(\beta) = - \frac{1}{2}\sum_{i=1}^n \lambda^{\alpha}_i \left(Y_i - X_i\beta \right)^2,
			$
			where $X_i=\left(X,e_i\right)_{L_2(\cX, \nu_\cX)}$,
			and thus
			$$
			\hat{\beta}_n =\frac{ \sum_{i=1}^n \lambda^{\alpha}_i  Y_i X_i}{ \sum_{i=1}^n \lambda^{\alpha}_i X^2_i}.
			$$
			First note that $Y_i = (\mathcal{S}X + \epsilon,e_i)_{L_2(\cX, \nu_\cX)}$, where $\epsilon$ can be written as $\epsilon = \cC^{1/2}\mathcal{W}$ for Gaussian white noise $\mathcal{W}$ on $L_2(\cX, \nu_\cX)$. Thus, 
			\begin{equation}
				\label{eq:proof_betahat}
			\hat{\beta}_n = \beta \frac{\sum_{i=1}^n \lambda_i^{\alpha}  X_i (\mathcal{S}X,e_i)_{L_2(\cX, \nu_\cX)} }{ \sum_{i=1}^n \lambda^{\alpha}_i X_i^2} + \frac{  \sum_{i=1}^n \lambda_i^{\alpha}  X_i (\mathcal{C}^{1/2}\mathcal{W},e_j)_{L_2(\cX, \nu_\cX)}}{ \sum_{i=1}^n \lambda^{\alpha}_i X_i^2}.
			\end{equation}
			The second term is a random variable with
			\begin{align*}
				\mathbb{V} \left[\frac{  \sum_{i=1}^n \lambda_i^{\alpha}  X_i (\mathcal{C}^{1/2}\mathcal{W},e_i)_{L_2(\cX, \nu_\cX)}}{ \sum_{i=1}^n \lambda^{\alpha}_i X_
					i^2} \right]&=
				\mathbb{V} \left[\frac{  \sum_{i=1}^n \lambda_i^{\alpha/2}  X_i \xi_i}{ \sum_{i=1}^n \lambda^{\alpha}_i X_i^2} \right] 
				= \frac{  \sum_{i=1}^n \lambda_i^{\alpha}  X_i^2 }{ \left( \sum_{i=1}^n \lambda^{\alpha}_i X_i^2 \right)^2}=B_n.
			\end{align*}
			Consider first the case $p < \alpha$:
			then, $\sum_{i=1}^\infty \lambda^{\alpha}_i X^2_i = \infty$ 
			and hence $B_n \rightarrow 0$. Thus, the asymptotic value of $\hat{\beta}_n$ depends solely on the first term in \eqref{eq:proof_betahat}:
			$$
			\beta\frac{\sum_{i=1}^n \lambda_i^{\alpha}  X_i (\mathcal{S}X,e_i)_{L_2(\cX, \nu_\cX)} }{ \sum_{i=1}^n \lambda^{\alpha}_i X^2_i} = 
			\beta\frac{ \sum_{i=1}^n \lambda_i^{\alpha}  \lambda_{S,i}^{\gamma}  X_i^2 }{ \sum_{i=1}^n \lambda^{\alpha}_i X^2_i} := \beta A_n.
			$$
			If $\gamma = 0$, we clearly have that $A_n \rightarrow 1$, so $\hat{\beta} \rightarrow \beta$. Further, 
			$$
			A_n \leq \frac{ \lambda_{S,n}^{\gamma} \sum_{i=1}^n \lambda_i^{\alpha} X_i^2 }{ \sum_{i=1}^n \lambda^{\alpha}_i X^2_i} = \lambda_{S,n}^{\gamma}. 
			$$
			Thus, $A_n \rightarrow 0$ if $\gamma <0$, so in this case $\hat{\beta}_n\rightarrow 0$. Now, take $m<n$, we then have that 
			$$
			A_n \geq 
			\frac{\sum_{i=1}^m \lambda_i^{\alpha}  \lambda_{S,i}^{\gamma}  X_i^2 + \lambda_{S,m}^{\gamma}\sum_{i=m}^n \lambda_i^{\alpha}   X_i^2}
			{\sum_{i=1}^m \lambda_i^{\alpha}  X_i^2 + \sum_{i=m}^n \lambda_i^{\alpha}   X_i^2}
			= 
			\frac{\frac{\sum_{i=1}^m \lambda_i^{\alpha}  \lambda_{S,i}^{\gamma}  X_i^2}{\sum_{i=m}^n \lambda_i^{\alpha}   X_i^2} + \lambda_{S,m}^{\gamma}}{\frac{\sum_{i=1}^m \lambda_i^{\alpha}  X_i^2}{\sum_{i=m}^n \lambda_i^{\alpha}   X_i^2} + 1} \rightarrow \lambda_{S,m}^{\gamma}
			$$
			as $n\rightarrow\infty$ because all $X_i$ are bounded and $\sum_{i=1}^{\infty} \lambda_i^{\alpha} X_i^2 = \infty$. Since this holds for any $m<n$, we can conclude that $A_n$ diverges if $\gamma>0$.
			Now consider the case $p\geq \alpha$. In this case, $\|X\|_{\cC}^2 = \sum_{i=1}^\infty \lambda^{\alpha}_i X_i^2 < \infty$. Further, 
			$$ 
			\sum_{i=1}^n \lambda^{\alpha}_i \lambda_{S,i}^{\gamma} X_i^2 \leq C_s \sum_{i=1}^n \lambda^{\alpha+\gamma}_i X_i^2 \rightarrow \|X\|_{A,\alpha+\gamma}^2 < \infty
			$$
			if $\alpha+\gamma \leq p$. Thus, in this case, 
			$
			\hat{\beta}_n \rightarrow (Y,X)_{\cC}/\|X\|_{\cC}^2 < \infty$ $\bbP$-a.s..
			If, on the other hand, $\alpha+\gamma > p$, 
			$
			\sum_{i=1}^n \lambda^{\alpha}_i \lambda_{S,i}^{\gamma} X_i^2  \rightarrow \infty
			$
			and thus, $\hat{\beta}_n \rightarrow sign(\beta)\infty$.
		\end{proof}
			
\end{appendix}

\bibliographystyle{chicago}
\bibliography{multivariate-bib}

\begin{thebibliography}{}

\bibitem[\protect\citeauthoryear{Bolin and Kirchner}{Bolin and Kirchner}{2020}]{BK2020rational}
Bolin, D. and K.~Kirchner (2020).
\newblock The rational {SPDE} approach for {G}aussian random fields with general smoothness.
\newblock {\em J. Comp. Graph. Stat.\/}~{\em 29\/}(2), 274--285.

\bibitem[\protect\citeauthoryear{Bolin and Kirchner}{Bolin and Kirchner}{2023}]{bk-measure}
Bolin, D. and K.~Kirchner (2023).
\newblock Equivalence of measures and asymptotically optimal linear prediction for {Gaussian} random fields with fractional-order covariance operators.
\newblock {\em Bernoulli\/}~{\em 29}, 1476--1504.

\bibitem[\protect\citeauthoryear{Bolin, Kirchner, and Kov\'{a}cs}{Bolin et~al.}{2020}]{BKK2020}
Bolin, D., K.~Kirchner, and M.~Kov\'{a}cs (2020).
\newblock Numerical solution of fractional elliptic stochastic {PDE}s with spatial white noise.
\newblock {\em IMA J. Numer. Anal.\/}~{\em 40\/}(2), 1051--1073.

\bibitem[\protect\citeauthoryear{Bolin, Kov\'{a}cs, Kumar, and Simas}{Bolin et~al.}{2023}]{bolinetal_fem_graph}
Bolin, D., M.~Kov\'{a}cs, V.~Kumar, and A.~B. Simas (2023).
\newblock Regularity and numerical approximation of fractional elliptic differential equations on compact metric graphs.
\newblock {\em Math. Comp.\/}.
\newblock In press.

\bibitem[\protect\citeauthoryear{Bolin, Simas, and Wallin}{Bolin et~al.}{2023}]{BSW2022}
Bolin, D., A.~B. Simas, and J.~Wallin (2023).
\newblock Gaussian {W}hittle-{M}at\'ern fields on metric graphs.
\newblock {\em Bernoulli\/}.
\newblock In press.

\bibitem[\protect\citeauthoryear{Borovitskiy, Terenin, Mostowsky, and Deisenroth}{Borovitskiy et~al.}{2020}]{borovitskiy2020mat}
Borovitskiy, V., A.~Terenin, P.~Mostowsky, and M.~Deisenroth (2020).
\newblock Mat\'{e}rn {G}aussian processes on {R}iemannian manifolds.
\newblock In H.~Larochelle, M.~Ranzato, R.~Hadsell, M.~Balcan, and H.~Lin (Eds.), {\em Advances in Neural Information Processing Systems}, Volume~33, pp.\  12426--12437. Curran Associates, Inc.

\bibitem[\protect\citeauthoryear{Cialenco}{Cialenco}{2018}]{cialenco2018}
Cialenco, I. (2018).
\newblock Statistical inference for {SPDEs}: an overview.
\newblock {\em Statistical Inference for Stochastic Processes\/}~{\em 21\/}(2), 309--329.

\bibitem[\protect\citeauthoryear{Cleveland}{Cleveland}{1979}]{cleveland1979}
Cleveland, W.~S. (1979).
\newblock Robust locally weighted regression and smoothing scatterplots.
\newblock {\em Journal of the American Statistical Association\/}~{\em 74\/}(368), 829--836.

\bibitem[\protect\citeauthoryear{Da~Prato and Zabczyk}{Da~Prato and Zabczyk}{2014}]{daPrato2014}
Da~Prato, G. and J.~Zabczyk (2014).
\newblock {\em Stochastic {E}quations in {I}nfinite {D}imensions\/} (Second ed.), Volume 152 of {\em Encyclopedia of Mathematics and its Applications}.
\newblock Cambridge University Press, Cambridge.

\bibitem[\protect\citeauthoryear{Dupont, Marques, and Kneib}{Dupont et~al.}{2023}]{dupont2023demystifying}
Dupont, E., I.~Marques, and T.~Kneib (2023).
\newblock Demystifying spatial confounding.
\newblock {\em arXiv preprint arXiv:2309.16861\/}.

\bibitem[\protect\citeauthoryear{Dupont, Wood, and Augustin}{Dupont et~al.}{2022}]{dupont2022spatial}
Dupont, E., S.~N. Wood, and N.~H. Augustin (2022).
\newblock Spatial+: a novel approach to spatial confounding.
\newblock {\em Biometrics\/}~{\em 78\/}(4), 1279--1290.

\bibitem[\protect\citeauthoryear{Genton and Kleiber}{Genton and Kleiber}{2015}]{crosscov}
Genton, M.~G. and W.~Kleiber (2015).
\newblock Cross-covariance functions for multivariate geostatistics.
\newblock {\em Statistical Science\/}~{\em 30\/}(2), 147 -- 163.

\bibitem[\protect\citeauthoryear{Gikhman and Skorokhod}{Gikhman and Skorokhod}{2004}]{Gikhman1974}
Gikhman, I.~I. and A.~V. Skorokhod (2004).
\newblock {\em The theory of stochastic processes. {I}}.
\newblock Classics in Mathematics. Springer-Verlag, Berlin.
\newblock Translated from the Russian by S. Kotz, Reprint of the 1974 edition.

\bibitem[\protect\citeauthoryear{Gilbert, Ogburn, and Datta}{Gilbert et~al.}{2023}]{gilbert2023consistency}
Gilbert, B., E.~L. Ogburn, and A.~Datta (2023).
\newblock Consistency of common spatial estimators under spatial confounding.
\newblock {\em arXiv preprint arXiv:2308.12181\/}.

\bibitem[\protect\citeauthoryear{Harbrecht, Herrmann, Kirchner, and Schwab}{Harbrecht et~al.}{2021}]{harbrecht2021multilevel}
Harbrecht, H., L.~Herrmann, K.~Kirchner, and C.~Schwab (2021).
\newblock Multilevel approximation of {Gaussian} random fields: Covariance compression, estimation and spatial prediction.
\newblock {\em arXiv preprint arXiv:2103.04424\/}.

\bibitem[\protect\citeauthoryear{Herrmann, Kirchner, and Schwab}{Herrmann et~al.}{2020}]{Herrmann2020}
Herrmann, L., K.~Kirchner, and C.~Schwab (2020).
\newblock Multilevel approximation of {G}aussian random fields: fast simulation.
\newblock {\em Math. Models Methods Appl. Sci.\/}~{\em 30\/}(1), 181--223.

\bibitem[\protect\citeauthoryear{Ibragimov and Rozanov}{Ibragimov and Rozanov}{2012}]{ibragimov2012gaussian}
Ibragimov, I.~A. and Y.~A. Rozanov (2012).
\newblock {\em Gaussian random processes}, Volume~9.
\newblock Springer Science \& Business Media.

\bibitem[\protect\citeauthoryear{Khan and Berrett}{Khan and Berrett}{2023}]{khan2023re}
Khan, K. and C.~Berrett (2023).
\newblock Re-thinking spatial confounding in spatial linear mixed models.
\newblock {\em arXiv preprint arXiv:2301.05743\/}.

\bibitem[\protect\citeauthoryear{Khan and Calder}{Khan and Calder}{2022}]{khan2022restricted}
Khan, K. and C.~A. Calder (2022).
\newblock Restricted spatial regression methods: Implications for inference.
\newblock {\em Journal of the American Statistical Association\/}~{\em 117\/}(537), 482--494.

\bibitem[\protect\citeauthoryear{Kirchner and Bolin}{Kirchner and Bolin}{2022}]{kirchner2022necessary}
Kirchner, K. and D.~Bolin (2022).
\newblock Necessary and sufficient conditions for asymptotically optimal linear prediction of random fields on compact metric spaces.
\newblock {\em The Annals of Statistics\/}~{\em 50\/}(2), 1038--1065.

\bibitem[\protect\citeauthoryear{Lehmann and Casella}{Lehmann and Casella}{2006}]{lehmann2006theory}
Lehmann, E.~L. and G.~Casella (2006).
\newblock {\em Theory of point estimation}.
\newblock Springer Science \& Business Media.

\bibitem[\protect\citeauthoryear{Lindgren, Bolin, and Rue}{Lindgren et~al.}{2022}]{lindgren2022spde}
Lindgren, F., D.~Bolin, and H.~Rue (2022).
\newblock The {SPDE} approach for {G}aussian and non-{G}aussian fields: 10 years and still running.
\newblock {\em Spat. Stat.\/}~{\em 50}, Paper No. 100599.

\bibitem[\protect\citeauthoryear{Lindgren, Rue, and Lindstr\"{o}m}{Lindgren et~al.}{2011}]{lindgren11}
Lindgren, F., H.~Rue, and J.~Lindstr\"{o}m (2011).
\newblock An explicit link between {G}aussian fields and {G}aussian {M}arkov random fields: the stochastic partial differential equation approach.
\newblock {\em J.\ R.\ Stat.\ Soc.\ Ser.\ B Stat.\ Methodol.\/}~{\em 73\/}(4), 423--498.
\newblock With discussion and a reply by the authors.

\bibitem[\protect\citeauthoryear{Paciorek}{Paciorek}{2010}]{paciorek2010importance}
Paciorek, C.~J. (2010).
\newblock The importance of scale for spatial-confounding bias and precision of spatial regression estimators.
\newblock {\em Statistical science: a review journal of the Institute of Mathematical Statistics\/}~{\em 25\/}(1), 107.

\bibitem[\protect\citeauthoryear{Reich, Hodges, and Zadnik}{Reich et~al.}{2006}]{reich2006effects}
Reich, B.~J., J.~S. Hodges, and V.~Zadnik (2006).
\newblock Effects of residual smoothing on the posterior of the fixed effects in disease-mapping models.
\newblock {\em Biometrics\/}~{\em 62\/}(4), 1197--1206.

\bibitem[\protect\citeauthoryear{Stein}{Stein}{1999}]{stein99}
Stein, M.~L. (1999).
\newblock {\em Interpolation of {S}patial {D}ata: {S}ome {T}heory for {K}riging}.
\newblock Springer Series in Statistics. Springer-Verlag, New York.

\bibitem[\protect\citeauthoryear{Zhang}{Zhang}{2004}]{Zhang2004}
Zhang, H. (2004).
\newblock Inconsistent estimation and asymptotically equal interpolations in model-based geostatistics.
\newblock {\em J.\ Amer.\ Statist.\ Assoc.\/}~{\em 99\/}(465), 250--261.

\end{thebibliography}
\end{document}